\documentclass[leqno,12pt]{amsart}

\usepackage{amsfonts,amssymb,amsgen,stmaryrd, amsopn}
\usepackage[leqno]{amsmath} 
\usepackage[T1]{fontenc}
\usepackage[latin1]{inputenc} 
\usepackage[french,english]{babel}
\usepackage{graphicx}
\usepackage{graphpap}
\usepackage{hyperref}
\newtheorem{theoreme}{Theorem}[section]

\newtheorem{proposition}{Proposition}[section]

 \newtheorem{rem}{Remark}[section]
 
\renewcommand\Re{\mathrm{Re}\,} \renewcommand\Im{\mathrm{Im}\,}
\newcommand\R{{\mathbb R}} \newcommand\N{{\mathbb N}}
 
\newcommand\C{{\mathbb C}} 
\newcommand\T{{\mathbb T}}

  \newcommand\s{\sigma}

\renewcommand{\Im}{  \text{Im}   }
\renewcommand{\Re}{  \text{Re}   }

 \def\cdotv{\raise 2pt\hbox{,}}

\renewcommand\Re{\mathrm{Re}\,} \renewcommand\Im{\mathrm{Im}\,}

\makeatletter
\def\@tvsp{\mathchoice{{}\mkern-4.5mu}{{}\mkern-4.5mu}{{}\mkern-2.5mu}{}}
\def\ltrivert{\left|\@tvsp\left|\@tvsp\left|}
\def\rtrivert{\right|\@tvsp\right|\@tvsp\right|}
\makeatother

 \def\cdotv{\raise 2pt\hbox{,}}

\begin{document}
\title{On the growth of Sobolev norms for NLS on $2d$ and $3d$ manifolds}
  \author{Fabrice Planchon}
  \address{    Universit\'e C\^ote d'Azur, CNRS, LJAD, France}
  \email{fabrice.planchon@unice.fr} 
\author{Nikolay Tzvetkov}
\address{
Universit\'e de Cergy-Pontoise,  Cergy-Pontoise, F-95000,UMR 8088 du CNRS
}
\email{nikolay.tzvetkov@u-cergy.fr}
\author{Nicola Visciglia}
\address{Dipartimento di Matematica, Universit\`a di Pisa, Italy}
\email{viscigli@dm.unipi.it}
\thanks{ The first author was
    partially supported by A.N.R. grant GEODISP and Institut Universitaire de France, the second author was partially supported by the ERC grant Dispeq, and 
    the third one was supported by the grant PRA 2016 Problemi di Evoluzione: Studio Qualitativo e Comportamento Asintotico} 
\date{\today}
 \maketitle
 \par \noindent

\begin{abstract}
Using suitable modified energies we study  higher order Sobolev norms' growth in time
for the nonlinear Schr\"odinger equation (NLS) on a generic $2d$ or $3d$ compact manifold.
In $2d$ we extend earlier results that dealt only with cubic nonlinearities, and get polynomial in time bounds for any higher order nonlinearities. In $3d$, we prove that solutions to the cubic NLS grow at most exponentially, while for sub-cubic NLS we get polynomial bounds on the growth of the 
$H^2$-norm.
\end{abstract}

\section{Introduction}

We are interested in long-time qualitative properties of solutions
to the following family of nonlinear Schr\"odinger equations,
\begin{equation}
  \label{eq:1}
  \begin{cases}
  i\partial_{t} u +\Delta_{g} u=|u|^{p-1} u, \quad (t, x)\in \R\times M^d\\
  u(0, x)=\varphi\in H^m(M^d)
  \end{cases}
\end{equation}
where $\Delta_g$ is the Laplace-Beltrami operator associated with the 
$d$-dimensional compact
Riemannian manifolds $(M^d, g)$ and $H^{m}(M^{d})$ the standard Sobolev space associated to $\Delta_{g}$, where $m\in \mathbb{N}$ with $m\geq 2$. More specifically we are interested in the analysis
of the possible growth of higher order Sobolev norms for large times, namely
the behavior of the quantity $\|u(t, x)\|_{H^m(M^d)}$ for $m\geq 2$ and $t\gg 1$.

This issue of the growth of higher order Sobolev norms has garnered a lot of attention in recent years, mainly because of its connection with the so called 
{\em weak wave turbulence}, e.g. a cascade of energy from low to high frequencies. In fact two main issues have been extensively studied in the literature: the first one concerns a priori bounds 
on how fast higher order Sobolev norms can grow along the flow associated with Hamiltonian PDEs
(see \cite{B3,B0,B,B2,CKO,D,GG,So1,So2,So3, S, JT, Z} and all the references therein); the second one concerns the existence of global solutions
whose higher order Sobolev norms are unbounded
(see \cite{CKSTT2, G, GK,H, HPTV,X} and all the references therein).

Here, we aim at dealing with the first problem, namely to provide a-priori
bounds on the growth of higher order Sobolev norms, or equivalently to understand 
how fast 
the dynamical system under consideration can move energy from the 
low frequencies to the high frequencies.  

First of all we point out that solutions to \eqref{eq:1} enjoy so-called mass and energy conservation laws:
\begin{align*}
\int_{M^d} |u(t, x)|^2 \hbox{dvol}_g&=\int_{M^d} |\varphi(x)|^2 \hbox{dvol}_g\\
\noindent\int_{M^d} (|\nabla_g u(t,x)|_g^2 + \frac 1{p+1} |u(t, x)|^{p+1})
\hbox{dvol}_g&=\int_{M^d} (|\nabla_g \varphi(x)|_g^2 + \frac 1{p+1} |\varphi(x)|^{p+1})
\hbox{dvol}_g
\end{align*}
where $\nabla_g$ and $|\hbox{ } . \hbox{ } |_g$ are respectively 
the gradient and the norm associated with the metric $g$, and
$|\hbox{ } . \hbox{ } |$ denotes the modulus of any complex number. These conservation laws immediately imply that 
\begin{equation}\label{H1tri}
\sup_{\R} \|u(t, x)\|_{H^1(M^d)}<\infty,\end{equation}
and therefore the growth in time of $H^m$ norms is only of interest for $m\geq 2$.\\
   
In the sequel, and according with the notations introduced above, 
we shall be interested in the following special cases:
\begin{align*}(d, p)&=(2, 2n+1) \hbox{ with } n\in \N, n\geq 1
\text{ ($2d$ manifold and odd integer nonlinearity)};\\\nonumber
(d, p)&=(3,3) \hbox{ ($3d$ manifold and cubic nonlinearity)};
\\\nonumber (d, p)&=(3,p) \hbox{ with } 2<p<3
\hbox{ ($3d$ manifold and sub-cubic nonlinearity)} .\end{align*}
In those setting, existence of local solutions follows by classical arguments, provided one assumes the initial datum to be $H^2$.
On the other hand, following \cite{BGT}, one can establish 
local (and hence global) Cauchy theory in $H^1$ for generic nonlinear potentials in the $2d$ case, as well as the local (and global) Cauchy theory  
in $H^{1+\epsilon}$ for cubic and sub cubic NLS
in the $3d$ case (see \cite{BGT} and \cite{BGT2}). From now on and for the sake of simplicity, we shall assume existence and uniqueness of a global solution, and focus on estimating the growth
of higher order Sobolev norms. However we point out that our argument not only provides polynomial bounds of such growth, but also yields an alternative proof of global existence. 

We will use as a basic tool (in fact, as a black box) available Strichartz estimates on manifolds (see \cite{BGT}, \cite{ST}) together with the introduction of suitable {\em modified energies},
which is the main new ingredient in this context. For this reason we will not discuss further the issue of global existence, which is indeed guaranteed by aforementioned previous results.
\\
We first start with the $2d$ case. It is worth mentioning that, to the authors' knowledge,
no results were available in the literature 
about growth of higher order Sobolev norms for NLS with higher than cubic nonlinearities, although one may reasonably believe that this problem could be addressed, at least in $2d$,
by adapting the strategy pioneered by Bourgain (see for instance \cite{Z}).
Nevertheless as a warm up we show how this problem can be handled 
by a completely different strategy, based on the introduction of suitable
{\em modified energies}: its benefit relies on a clear decoupling between higher order energy estimates relying on clever integration by parts and the (deep) input provided by dispersive estimates of Strichartz type. Moreover by using modified energies one can deal as well 
with generic nonlinear potential $V(|u|^2)$ rather than $|u|^{p-1}$, where $V$ may not necessarily be a pure power (see also remark \ref{pol3d} below).

We emphasize that modified energies have proved useful 
in different contexts (see for instance \cite{CR}, \cite{HITW},  \cite{KT}, \cite{Kw},
\cite{RS}, \cite{Ts}), but the present work seems to provide the first example where they are combined with dispersive bounds in order to get results on the growth of 
higher order Sobolev norms.

We underline that in our argument, being essentially based on integration by parts, we define and then compute the time derivative of suitable higher order energies
${\mathcal E}_m$, whose leading term
is essentially the norm $\|u(t,x) \|_{H^m}^2$. In fact, for $m=2k$ an even integer, one should think of $\|\partial_{t}^{k} u(t,x)\|^{2}_{L^{2}}$ as a good prototype of modified energy, up to lower order terms. In other words, one should think of replacing $\Delta_{g}$ by $\partial_{t}$ rather than the other way around when using the equation satisfied by $u$.

A direct consequence of this priviledged use of $\partial_{t}$ is that in our approach the geometry of the manifold 
is not directly involved in the computation, and integration by parts in the space variables, when required, is performed 
thanks to the following elementary identity, available
on any generic manifold:
$$\Delta_g (fh)=h\Delta_g f + 2 (\nabla_g f, \nabla_g h)_g + f\Delta_g h.$$
We also underline that the aforementioned energy ${\mathcal E}_m$
is not preserved along the flow, however by computing its time derivative along solutions,
we may estimate the resulting space-time integral taking advantage of dispersive bounds, namely  Strichartz estimates
with loss which are available on a generic manifold (or better ones when available).
\\
\\
In order to state our result in $2d$ we recall Strichartz estimates with loss:
\begin{equation}\label{strichintro}
\|e^{it \Delta_{g}} \varphi\|_{L^4((0,1) \times M^2)}
\lesssim \|\varphi\|_{H^{s_0}(M^2)}.
\end{equation}
For every $2d$ compact manifold $M^2$, \eqref{strichintro} is known to hold for $s_0= \frac 14$ (see \cite{BGT}).
Of course in some special cases the previous bound may be improved,
for instance 
on $\T^2$ one has any $s_0>0$ (see \cite{B3}, \cite{B0})
and any $s_0>\frac 18$ on the sphere (see \cite{BGT}). We can now state our first result.
\begin{theoreme}\label{main2d}
For every $\epsilon>0$, $m\in \N$ and for every $u(t,x)\in {\mathcal C}_t (H^m(M^2))$ solution to \eqref{eq:1}
where $d=2$ and $p=2n+1$ for $n\geq 1$, we get
\begin{equation}\label{growthm}
\sup_{(0,T)} \|u(t,x)\|_{H^m(M^2)} \leq C T^{\frac{m-1}{1-2s_0}+\epsilon},
\end{equation}
where $C=C(\epsilon,m, \|\varphi\|_{H^m})>0$ and $s_0\geq 0$ is given
in \eqref{strichintro}.
\end{theoreme}
\begin{rem}\label{2dpolite}
We underline that the main point in order to establish
Theorem \ref{main2d} is the following bound: for all $T\in (0,1)$, $\epsilon>0$,
\begin{equation}\label{polite}
\|u(T)\|_{H^{m}(M^2)}^2 - \|u(0)\|_{H^{m}(M^2)}^2 \lesssim \sqrt T 
\|u\|_{L^\infty_T H^{m}(M^2)}^{\frac{2m-3+2s_0}{m-1} +\epsilon}+
\|u\|_{L^\infty_T H^{m}(M^2)}^{\frac{2m-4}{m-1}+\epsilon}\,.
\end{equation}
Once this bound is established then polynomial growth follows by a 
straightforward iteration argument. More precisely notice that the exponent 
$\frac{m-1}{1-2s_0}+\epsilon$ (that appears in the r.h.s of \eqref{growthm})
can be computed as the quantity
$\frac 1{2\gamma}$, where $2-2\gamma= \frac{2m-3+2s_0}{m-1} +\epsilon
$ is the power of the first term in the r.h.s. 
of \eqref{polite}.
\end{rem}

Next we present our result on the growth of higher order Sobolev norms
for the cubic NLS on a generic $3d$ compact manifold $M^3$. We recall that, following
\cite{BGT}, the Cauchy problem is globally well-posed for every initial data
$\varphi\in H^{1+\epsilon_0}(M^3)$, and that, following the crucial use of logarithmic Soboleve type inequalities,  one can get the following double exponential bound,
$$\sup_{(0,T)} \|u(t, x)\|_{H^m(M^3)}\leq C \exp(\exp (CT)).$$
Our main contribution is an improvement on the bound above,
indeed we will replace the double exponential with a single one.
It should be emphasized that, in the $3d$ case, it is at best unclear to us how Bourgain original argument and derivatives thereof could be used in order to get
Theorem \ref{main3d}. More specifically, in $3d$ our use of modified energies appears to be a key tool in order to eliminate one exponential out of two.
\begin{theoreme}\label{main3d}
For every $m\in \N$ and for every $u(t,x)\in {\mathcal C}_t (H^m(M^3))$ solution to \eqref{eq:1}
where $(d,p)=(3,3)$ 
we have:
$$\sup_{(0,T)} \|u(t, x)\|_{H^m(M^3)}\leq C \exp (CT)$$
where $C=C(m, \|\varphi\|_{H^m})>0$.
\end{theoreme}
\begin{rem}
The proof of Theorem \ref{main3d} follows by a straightforward
iteration once the following bound is established: for all $ T\in (0,1)$,
$$\|u(T)\|_{H^{m}(M^3)}^2 - \|u(0)\|_{H^{m}(M^3)}^2 \lesssim 
T \|u\|_{L^\infty_T H^{m}(M^3)}^2 + \|u\|_{L^\infty_T H^{m}(M^3)}^\gamma,
$$
where $\gamma\in (0,2)$ is a suitable number.
\end{rem}

Finally, we end our presentation with a result dealing with NLS on a $3d$ compact manifold 
$M^3$ with sub-cubic nonlinearity, establishing  polynomial growth for the $H^2$ Sobolev norm.
Remark that it makes no sense to consider higher order Sobolev norms, given that the nonlinearity is not smooth enough in order to guarantee that the regularity $H^m$, with $m>2$, is preserved along the evolution.

Neverhtless we emphasize that the next result appears to be the first one available in the literature about polynomial growth of any  Sobolev norms above the energy, on a generic $3d$ compact manifold.
\begin{theoreme}\label{main3dsub} 
For every $u(t,x)\in {\mathcal C}_t (H^2(M^2))$ solution to \eqref{eq:1}
with $d=3$ and $p\in (2, 3)$ we have:
$$\sup_{(0,T)} \|u(t, x)\|_{H^2(M^3)}\leq CT^{\frac 4{3-p}}\,,$$
where $C=C(\|\varphi\|_{H^2})>0$.
\end{theoreme}
\begin{rem}\label{pol3d}
Analogously to the theorems above this result follows once
the following local bound is established: for all $T\in (0,1)$,
$$\|u(T)\|_{H^{2}(M^3)}^2 - \|u(0)\|_{H^{2}(M^3)}^2
\lesssim 
T\|u\|_{L^\infty_T H^2(M^3)}^\frac{p+5}4+ \|u\|_{L^\infty_T H^2(M^3)}^\gamma,
$$
for some $\gamma\in (0, \frac{p+5}4)$. One should point out that, following our approach to proving \eqref{pol3d}, there is no need to restrict oneself to pure power nonlinearities. In particular polynomial growth for solutions to NLS on a generic $3d$ compact manifolds 
could be established for general higher order Sobolev norms (namely $H^m$ with $m\geq 2$), provided that the sub cubic nonlinearity is suitably regularized in order to guarantee that $H^m$ regularity is preserved along the flow. Neverthless, for the sake of simplicity we elected not to deal with the full generality in the present work.
\end{rem}

\section{Linear Strichartz Estimates}
\subsection{Strichartz Estimates on $M^2$}
In the sequel we shall make use without any further comment 
of the following Strichartz estimate which was already recalled in the introduction:
\begin{equation}\label{strich}
\|e^{it \Delta_{g}} \varphi\|_{L^4((0,1) \times M^2)}\lesssim
\|\varphi\|_{H^{s_0}(M^2)}.
\end{equation}
By using Duhamel formula we also have at our disposal an inhomogeneous estimate that we state as an independent proposition.
\begin{proposition}
Let $v(t,x)$ be solution to
$$\begin{cases}
i\partial_t v + \Delta_g v=F, \quad (t, x)\in \R\times M^2\\
u(0, x)=\varphi\in H^{s_0}(M^2)
\end{cases}$$
then we have, for $T\in (0,1)$
\begin{equation}\label{duhamel}
\|v\|_{L^4((0,T) \times M^2)}\lesssim 
\|\varphi\|_{H^{s_0}(M^2)}+ T \|F\|_{L^\infty((0,T); H^{s_0}(M^2))}.
\end{equation}
\end{proposition}
\subsection{Strichartz Estimates on $M^3$}
Along the proof of Theorems \ref{main3d} and 
\ref{main3dsub} we shall make use of the following
suitable version of the endpoint Strichartz estimate:
\begin{proposition}
Let $v(t,x)$ be solution to 
$$
i\partial_t v + \Delta_g v=F, \quad (t, x)\in \R\times M^3.
$$
Then we have, for $T\in (0,1)$,
\begin{equation}\label{duhamel3}
\|v\|_{L^2((0,T); L^6(M^3))}\lesssim_{\epsilon} 
\|v\|_{L^\infty((0,T);H^{\epsilon}(M^3))}+
\|v\|_{L^2((0,T);H^{1/2}(M^3))}+ \|F\|_{L^2(0,T); L^{6/5}(M^3))}\,.
\end{equation}
\end{proposition}
Notice that the above estimate may look somewhat unusual compared with the classical version of Strichartz
estimates, where on the r.h.s. one expects a norm involving
the initial datum $v(0,x)$ and another norm involving 
the forcing term $F(t,x)$.\\
Neverthless we underline 
that in the case $F=0$, the estimate above reduces to the usual
Strichartz estimate with loss of $1/2$ derivative (see \cite{BGT} and \cite{ST}). On the other hand the main point of \eqref{duhamel3} is that no derivative losses occur on the forcing term $F(t,x)$
when this term is not identically zero, and the loss of derivative is indeed absorbed 
by the solution $v(t,x)$. Estimates of this spirit are also of crucial importance in the low regularity well-posedness theory for quasi-linear dispersive PDE (see e.g. \cite{KTz}).
We emphasize that the estimate \eqref{duhamel3} comes
from the following spectrally localized version (see \cite{BGT}, \cite{ST} and for more details
Proposition 5.4 in \cite{BT}):
\begin{align*}
\|\pi_N v\|_{L^2((0,1); L^6(M^3))}&\lesssim \|\pi_N v\|_{L^\infty(0,1);L^{2}(M^3)}
\\\nonumber &+\|\pi_N v\|_{L^2(0,1);H^{1/2}(M^3)}
+\|\pi_N F\|_{L^2(0,1); L^{6/5}(M^3))}
\end{align*}
where $(\pi_N)$ are the usual Littlewood-Paley spectral projectors
and $N$ ranges over dyadic numbers.
In fact by taking squares and summing over $N$ we get \eqref{duhamel3},
provided that we make use of the following bound,
$$\sum_N \|\pi_N v\|_{L^\infty ((0,1);L^2(M^3))}^2\leq \sum_N \frac 1{N^\epsilon} \|v\|_{L^\infty( (0,1);H^{\epsilon}(M^3))}^2$$
together with the equivalence of the $L^r$ norm of $v$ with the $L^r$ ($1<r<\infty$) norm of its squared function 
$(\sum_N |\pi_N v|^2)^\frac 12$. 
\section{Modified Energies associated with Even Sobolev norms}
\subsection{Modified Energies}
In this subsection we consider the general Cauchy problem
\begin{equation}\label{genericNLS}
\begin{cases}
i\partial_{t} u +\Delta_{g} u=|u|^{p-1} u,
\quad (t, x)\in \R\times M^d\\
u(0, x)=\varphi\in H^{2k}(M^d)
\end{cases}
\end{equation}
where $(M^d, g)$ is a compact $d$-dimensional Riemannian manifold.

In the sequel we shall extensively make use of the following bound without further notice:
\begin{equation}\label{H1ener}
\|u\|_{L^\infty(\R; H^1(M^d))}\lesssim_{_{p,\|\varphi\|_{H^1}}}1 \,.
\end{equation}
For every solution $u(t, x)$ to the Cauchy problem \eqref{genericNLS}
we introduce the following energy, to be used in connection with growth of the Sobolev norm $H^{2k}$:
\begin{align*}{\mathcal E}_{2k}(u)&=
\|\partial_t^k u\|_{L^2(M^d)}^2-\frac{p-1}4 \int_{M^d} |\partial^{k-1}_{t} \nabla_g (|u|^2)|_g^2 |u|^{p-3} \hbox{dvol}_g \\\nonumber&- \int_{M^d}|\partial^{k-1}_{t} (|u|^{p-1} u)|^2 \hbox{dvol}_g.
\end{align*}
We have the following key proposition.
\begin{proposition}\label{energy}
Let $u(t, x)$ be solution to \eqref{genericNLS} where $p=2n+1\geq 3$,
with initial data $\varphi\in H^{2k}(M^d)$,
then we have the following identity:
\begin{multline}\label{energy2}
\frac d{dt} {\mathcal E}_{2k}(u(t,x))=
-\frac{p-1}4 \int_{M^d} |\partial^{k-1}_{t} \nabla_g (|u|^2)|_g^2 \partial_t (|u|^{p-3}) \hbox{dvol}_g \\
{}+ 2  \int_{M^d} \partial^{k}_{t} (|u|^{p-1})\partial^{k-1}_{t}
(|\nabla_g u|_g^2) \hbox{dvol}_g
\\{}+ \sum_{j=0}^{k-1} c_j \int_{M^d} (\partial^{j}_{t} \nabla_g (|u|^2),
\partial^{k-1}_{t}  \nabla_g (|u|^2)_g \partial_t^{k-j} (|u|^{p-3}) \hbox{dvol}_g 
\\
{} + \Re \sum_{j=0}^{k-1} c_j    \int_{M^d}   
\partial^{j}_{t} (|u|^{p-1})  \partial_t^{k-j} u \partial_{t}^{k-1 }(|u|^{p-1}\bar u)
\hbox{dvol}_g,
\\
{}+ {\Re} \sum_{j=0}^{k-2}
c_j \int_{M^d} \partial^{k}_{t} (|u|^{p-1}) \partial_t^j (\Delta_g \bar u) \partial_t^{k-1-j} u
\hbox{dvol}_g
\\{}+ \Im \sum_{j=1}^{k-1} 
 c_j  \int_{M^d} \partial_t^j (|u|^{p-1}) \partial_t^{k-j} u \partial^{k}_{t}\bar u \hbox{dvol}_g
\end{multline}
where $c_j$ denote suitable constants that may change from line to line.
\end{proposition}
\begin{proof}
We start with the following computation:
\begin{align*}
\frac{d}{dt} \| \partial_{t}^{k} u\|^{2}_{L^2(M^d)} & = 2 \Re (\partial^{k+1}_{t} u , \partial^{k}_{t} u) = 2 \Re (\partial^{k}_{t} (-\Delta_g u+|u|^{p-1} u), i\partial^{k}_{t} u) \\
& = 2 \Im \int_{M^d} (\partial^{k}_{t} \nabla_g u, \partial^{k}_{t} \nabla_g u)_g
\hbox{dvol}_g
+ 2\Re(\partial^{k}_{t} (|u|^{p-1} u),i\partial^{k}_{t}u)
\end{align*}
where $(f,g)$ denotes the usual $L^{2}(M^{l})$ scalar product $\int_{M^d} f \cdot \bar g \hbox{dvol}_g$.
Since the first term on the r.h.s. vanishes we get 
\begin{align*}
\frac{d}{dt} \| \partial^{k}_{t} u\|^{2}_{L^2(M^d)} & =  2\Re(\partial^{k}_{t} (|u|^{p-1} u),i\partial^{k}_{t}u)
   =  2\Re(\partial^{k}_{t} (|u|^{p-1})  u,i\partial^{k}_{t}u)\\&+ 2\Re(|u|^{p-1} \partial^{k}_{t}  u,i\partial^{k}_{t}u)+ \Re \sum_{j=1}^{k-1} 
 c_j ( \partial_t^j (|u|^{p-1}) \partial_t^{k-j} u, i\partial^{k}_{t}u)
\end{align*}
where $c_j$ are suitable real numbers. 
Notice that the second term on the r.h.s. vanishes and if we
substitute for the equation again then we get:
\begin{align}\label{ImP}
&\frac{d}{dt} \|\partial^{k}_{t} u\|^{2}_{L^2(M^d)}  =  2\Re(\partial^{k}_{t} (|u|^{p-1})  u,-\Delta_g (\partial^{k-1}_{t}u))\\\nonumber&+2\Re(\partial^{k}_{t} (|u|^{p-1})  u,\partial_{t}^{k-1 }(|u|^{p-1}u))+
\Re \sum_{j=1}^{k-1} 
 c_j ( \partial_t^j |u|^{p-1} \partial_t^{k-j} u, i\partial^{k}_{t}u)\\\nonumber
 & =  2\Re(\partial^{k}_{t} (|u|^{p-1})  u,-\Delta_g (\partial^{k-1}_{t}u))
 +2\Re(\partial^{k}_{t} (|u|^{p-1} u),\partial_{t}^{k-1 }(|u|^{p-1}u))\\\nonumber &
 + \Re \sum_{j=0}^{k-1} c_j  (  
\partial^{j}_{t} (|u|^{p-1})  \partial_t^{k-j} u,\partial_{t}^{k-1 }(|u|^{p-1}u)) + \Re
\sum_{j=1}^{k-1} 
 c_j ( \partial_t^j |u|^{p-1} \partial_t^{k-j} u, i\partial^{k}_{t}u)
 \\\nonumber
 & =  2\Re(\partial^{k}_{t} (|u|^{p-1})  u,-\Delta_g (\partial^{k-1}_{t}u))
 + \int_{M^d} \partial_t |\partial^{k-1}_{t} (|u|^{p-1} u)|^2 \hbox{dvol}_g\\\nonumber &
 + \Re \sum_{j=0}^{k-1} c_j  (  
\partial^{j}_{t} (|u|^{p-1})  \partial_t^{k-j} u,\partial_{t}^{k-1 }(|u|^{p-1}u)) +\Re
\sum_{j=1}^{k-1} 
 c_j ( \partial_t^j (|u|^{p-1}) \partial_t^{k-j} u, i\partial^{k}_{t}u).
\end{align}
Next we focus on the first term on the r.h.s.
$$2\Re(\partial^{k}_{t} (|u|^{p-1})  u,-\Delta_g (\partial^{k-1}_{t}u))=
\int_{M^d} \partial^{k}_{t} (|u|^{p-1}) (-\bar u \partial_t^{k-1} (\Delta_g u) -u \partial_t^{k-1} 
(\Delta_g \bar u)) \hbox{dvol}_g$$
and we notice  \begin{align*}
-\bar u \Delta_g (\partial^{k-1}_{t}u)  - u\Delta_g (\partial^{k-1}_{t}\bar u)
=\partial_{t}^{k-1}(-\bar u \Delta_g u-u\Delta_g \bar u)
+ \Re \sum_{j=0}^{k-2}
c_j \partial_t^j (\Delta_g u) \partial_t^{k-1-j} \bar u.\end{align*}
Moreover we have the identity
$$
\Delta_g (|u|^2)= u \Delta_g \bar u+\bar u \Delta_g u+ 2 |\nabla_g u|_g^2
$$
hence we get
\begin{align*}2\Re(\partial^{k}_{t} &(|u|^{p-1})  u,-\Delta_g \partial^{k-1}_{t}u)= - \int_{M^d} \partial^{k}_{t} (|u|^{p-1}) \partial^{k-1}_{t} \Delta_g (|u|^2) 
\hbox{dvol}_g\\&+ 2  \int_{M^d} \partial^{k}_{t} (|u|^{p-1})\partial^{k-1}_{t}
(|\nabla_g u|_g^2) \hbox{dvol}_g + \Re \sum_{j=0}^{k-2}
c_j   \int_{M^d} \partial^{k}_{t} (|u|^{p-1}) \partial_t^j (\Delta_g u) \partial_t^{k-1-j} \bar u 
\hbox{dvol}_g \\
&=\int_{M^d} (\partial^{k}_{t} \nabla_g (|u|^{p-1}), \partial^{k-1}_{t} \nabla_g (|u|^2))_g
\hbox{dvol}_g
+ 2  \int_{M^d} \partial^{k}_{t} (|u|^{p-1})\partial^{k-1}_{t}
(|\nabla_g u|_g^2)\hbox{dvol}_g
\\
&+ \Re \sum_{j=0}^{k-2}
c_j   \int_{M^d} \partial^{k}_{t} (|u|^{p-1}) \partial_t^j (\Delta_g u) \partial_t^{k-1-j} \bar u
\hbox{dvol}_g 
\end{align*}
and by elementary computations we get
\begin{align*}
...&=\frac{p-1}2 \int_{M^d} (\partial^{k}_{t}( \nabla_g (|u|^2) |u|^{p-3}), \partial^{k-1}_{t} \nabla_g (|u|^2))_g
\hbox{dvol}_g
+ 2  \int_{M^d} \partial^{k}_{t} (|u|^{p-1})\partial^{k-1}_{t}
(|\nabla_g u|_g^2)\hbox{dvol}_g
\\
&+\Re \sum_{j=0}^{k-2}
c_j   \int_{M^d} \partial^{k}_{t} (|u|^{p-1}) \partial_t^j (\Delta_g u) \partial_t^{k-1-j} \bar u 
\hbox{dvol}_g\,.\end{align*}
Using Leibnitz rule to develop $\partial^{k}_{t}$ we get
\begin{align*}...&=\frac{p-1}2 \int_{M^d} (\partial^{k}_{t} \nabla_g (|u|^2) |u|^{p-3},
\partial^{k-1}_{t} \nabla_g (|u|^2))_g
\hbox{dvol}_g
\\&+ \sum_{j=0}^{k-1} c_j \int_{M^d}(\partial^{j}_{t} \nabla_g (|u|^2),
\partial^{k-1}_{t} \nabla_g (|u|^2))_g \partial_t^{k-j} (|u|^{p-3}) \hbox{dvol}_g\\&
+ 2  \int_{M^d} \partial^{k}_{t} (|u|^{p-1})\partial^{k-1}_{t}
(|\nabla_g u|_g^2) \hbox{dvol}_g
+ \Re \sum_{j=0}^{k-2}
c_j   \int_{M^d} \partial^{k}_{t} (|u|^{p-1}) \partial_t^j (\Delta_g u) \partial_t^{k-1-j} \bar u \hbox{dvol}_g
\\\nonumber
&= \frac{p-1}4 \int_{M^d} \partial_t |\partial^{k-1}_{t} \nabla_g (|u|^2)|_g^2 |u|^{p-3}
\hbox{dvol}_g
\\\nonumber &
+ \sum_{j=0}^{k-1} c_j \int_{M^d}
(\partial^{j}_{t} \nabla_g (|u|^2),\partial^{k-1}_{t} \nabla_g (|u|^2))_g \partial_t^{k-j} (|u|^{p-3})\hbox{dvol}_g
\\&+ 2  \int_{M^d} \partial^{k}_{t} (|u|^{p-1})\partial^{k-1}_{t}
(|\nabla_g u|_g^2) \hbox{dvol}_g+  \Re \sum_{j=0}^{k-2}
c_j  \int_{M^d} \partial^{k}_{t} (|u|^{p-1}) \partial_t^j (\Delta_g u) \partial_t^{k-1-j} \bar u \hbox{dvol}_g\,,
\end{align*}
and we conclude by combining this identity with \eqref{ImP}.
\end{proof}

\begin{rem}
In the specific case of cubic NLS (i.e. \eqref{genericNLS} with $p=3$) 
we have some simplifications, more precisely we get:
\begin{align*}{\mathcal E}_{2k}(u)&=
\|\partial_t^k u\|_{L^2(M^d)}^2-\frac{1}2 \int_{M^d} |\partial^{k-1}_{t} \nabla_g (|u|^2)|_g^2 \hbox{dvol}_g - \int_{M^d}|\partial^{k-1}_{t} (|u|^{2} u)|^2 \hbox{dvol}_g
\end{align*}
and also
\begin{align}\label{energy3}
\frac d{dt} {\mathcal E}_{2k}(u(t,x))&=2  \int_{M^d} \partial^{k}_{t} (|u|^{2})\partial^{k-1}_{t}
(|\nabla_g u|_g^2)\hbox{dvol}_g
\\\nonumber
&+\Re \sum_{j=0}^{k-2}
c_j   \int_{M^d} \partial^{k}_{t} (|u|^{2}) \partial_t^j (\Delta_g u) \partial_t^{k-1-j} \bar u 
\hbox{dvol}_g
\\\nonumber &+ \Re \sum_{j=0}^{k-1} c_j  \int_{M^d}  
\partial^{j}_{t} (|u|^{2})  \partial_t^{k-j} u \partial_{t}^{k-1 }(|u|^{2}\bar u) \hbox{dvol}_g 
\\\nonumber & +\Im
\sum_{j=1}^{k-1} 
 c_j \int_{M^d} \partial_t^j (|u|^{2}) \partial_t^{k-j} u \partial^{k}_{t}\bar u \hbox{dvol}_g.
\end{align}
\end{rem}
\subsection{The norms $\|\partial_t^k u\|_{L^2}$ and 
$\|u\|_{H^{2k}}$ are comparable.
}
The aim of this subsection is indeed to prove that the leading term
in our modified energy $\mathcal E_{2k} (u)$
is equivalent to the Sobolev norm $\|u\|_{H^{2k}}$, provided that $u(t,x)$ is a
solution to
\eqref{genericNLS} with $d=2$ and $p\geq 3$
or $d=3$ and $p=3$. 

\begin{proposition}\label{equiBNN}
Let $u(t,x)$ be solution to \eqref{genericNLS}, where
either $d=2$ and $p\geq 3$ is an integer, or $d=3$ and $p=3$.
Then for every $k,s\in \N$ we have:
\begin{align}\label{eqENimproved2d}
\| \partial_t^k u- i^k \Delta_g^k u\|_{H^s(M^d)}
\lesssim_{\|\varphi\|_{H^1}}  \|u\|_{H^{s+2k-1}(M^d)}\,.
\end{align}
\end{proposition}

\begin{proof}
We shall use the following identity (satisfied by every solution to \eqref{genericNLS}
in any dimension $d$):
\begin{equation}\label{basicequiv}
\partial_t^h u= i^h \Delta_g^h u + \sum_{j=0}^{h-1} c_j \partial_t^j \Delta_g^{h-j-1} (u|u|^{p-1})
\end{equation}
where $c_j\in \C$ are suitable coefficients. The elementary
proof follows by induction on $h$ and by using the equation
solved by $u(t,x)$.

{\em First case: $d=2$, $p\geq 3$}. We argue by induction on $k$, and hence we shall prove $k\Rightarrow k+1$.
By \eqref{basicequiv} we aim at proving
\begin{equation}\label{nosal}\|\partial_t^j (u|u|^{p-1})\|_{H^{k-j+s}(M^2)} \lesssim 
\|u\|_{H^{s+2k+1}(M^2)}, \quad j=0,..,k,\end{equation}
by assuming the property \eqref{eqENimproved2d}
true for $k$.
By expanding the time and space derivatives on the l.h.s.
above, we deduce \eqref{nosal} by the following
chain of inequalities:
\begin{align*}
\prod_{\substack{j_1+...+j_p=j\\s_1+...+s_p=k-j+s}}&\|\partial_t^{j_l} u\|_{W^{s_l, 2p}(M^2)}
\lesssim \prod_{\substack{j_1+...+j_p=j\\s_1+...+s_p=k-j+s}}
\|\partial_t^{j_l} u\|_{H^{s_l+1}(M^2)}\\\nonumber& \lesssim 
\prod_{\substack{j_1+...+j_p=j\\s_1+...+s_p=k-j+s}}\|u
\|_{H^{2j_l+s_l+1}(M^2)}
\end{align*}
where we used the Sobolev embedding $H^1(M^2)\subset L^{2p}(M^2)$
and
we have used the induction hypothesis at the last step.
We can continue the estimate by a trivial interpolation argument as follows:
$$...\lesssim 
\big( \prod_{k=1,...,p}  \|u\|_{H^{s+2k+1}(M^2)}^{\theta_l}\|u\|_{H^{1}(M^2)}^{(1-\theta_l)}
\big)$$
where 
$$\theta_l(s+2k+1)+ (1-\theta_l)= 2j_l+s_l+1.$$
We conclude using \eqref{H1ener}, since $\sum_{l=1}^p \theta_l=\frac{j+k+s}{s+2k}\leq 1$ for $j=0,...,k$.

{\em Second case: $d=3$, $p=3$}. Arguing as above, and by assuming the result true for $k$,
then we are reduced to proving
$$\|\partial_t^j (u|u|^{2})\|_{H^{k-j+s}(M^3)}\lesssim \|u\|_{H^{s+2k+1}(M^3)}, 
\quad j=0,...,k.$$
Expanding again the time and space derivatives on the l.h.s.,
we are reduced to the following estimate:
\begin{align*}\|\partial_t^{j_1} u \|_{W^{k_1, 6}(M^3)} &\times \| \partial_t^{j_2} u\|_{W^{k_2, 6}(M^3)}
\times 
\|\partial_t^{j_3} u\|_{W^{k_3, 6}(M^3)}\\\nonumber &\lesssim 
\|\partial_t^{j_1} u \|_{H^{k_1+1}(M^3)} \times \| \partial_t^{j_2} u\|_{H^{k_2+1}(M^3)}
\times 
\|\partial_t^{j_3} u\|_{H^{k_3+1}(M^3)} \\\nonumber
&\lesssim \|u\|_{H^{2j_1+k_1+1}(M^3)}
\|u\|_{H^{2j_2+k_2+1}(M^3)}\|u\|_{H^{2j_3+k_3+1}(M^3)}
\end{align*}
where
$$\begin{cases}j_1+j_2+j_3=j
\\k_1+k_2+k_3=k-j+s.
\end{cases}$$
Notice that we have used the Sobolev embedding $H^1(M^3)\subset L^6(M^3)$
and the induction hypothesis at the last step.
By interpolation we have
$$\|u\|_{H^{2j_l+k_l+1}(M^3)}\lesssim \|u\|_{H^{s+2k+1}(M^3)}^{\theta_l}\|u\|_{H^{1}(M^3)
}^{1-\theta_l}, \quad l=1,2,3,
$$
where 
$$2j_l+k_l+1=1-\theta_l+\theta_l(s+2k+1)$$
and we conclude as above since
$\sum_{l=1}^3 \theta_l=\frac{j+k+s}{s+2k}\leq 1$ for $j=0,...,k$.
\end{proof}
\subsection{Strichartz Estimates for Nonlinear Solutions}
In this subsection we get a priori bounds for the Strichartz 
norms of solutions to \eqref{genericNLS} 
in dimension $d=2$, with a general nonlinearity, and in dimension $d=3$, with cubic nonlinearity. 
\begin{proposition}\label{theorStrichartz2}
We have the following estimate for every $u(t, x)$ solution to \eqref{genericNLS}
for $d=2$ and $p=2n+1\geq 3$ is an integer: for any $\epsilon>0$ and $T\in (0,1)$,
\begin{multline}
\|\partial_t^j u\|_{L^4_TW^{s,4}(M^2)} \lesssim_{_{\epsilon,\|\varphi\|_{H^{1}}}}  
\|u\|_{L^\infty_T H^{2j+s}(M^2)}^{1-s_0} \|u\|_{L^\infty_T H^{2j+s+1}(M^2)}^{s_0} 
 \|u\|_{L^\infty_T H^{2j+2}(M^2)}^{\epsilon}\,.
\end{multline}
\end{proposition}
\begin{proof}
We use \eqref{duhamel},
together with 
the equation solved by $\partial_t^j u$, and we get:
\begin{align*}
\|\partial_t^j u\|_{L^4_TW^{s,4}(M^2)}  \lesssim & \|\partial_t^j u(0)\|_{H^{s+s_0}(M^2)}
+ T \|\partial_t^j (u|u|^{p-1})\|_{L^\infty_T H^{s+ s_0}(M^2)}\\
  \lesssim & \|\partial_t^j u(0)\|_{H^{s}}^{1-s_0} \|\partial_t^j u(0)\|_{H^{s+1}(M^2)}^{s_0}
\\
 & {}+ T \|\partial_t^j (u|u|^{p-1})\|_{L^\infty_T H^{s}(M^2)}^{1-s_0}\|\partial_t^j 
(u|u|^{p-1})\|_{L^\infty_T H^{s+ 1}(M^2)}^{s_0}
\end{align*}
Notice that the first term on the r.h.s. can be estimated by 
Proposition \ref{equiBNN}.
Hence we shall complete the proof 
provided that for every $\epsilon>0$,
$$\|\partial_t^j (u|u|^{p-1})\|_{H^{s}(M^2)}
\lesssim_{_{\epsilon, \|\varphi\|_{H^1}}} \|u\|_{H^{2j+s}(M^2)} \|u\|_{H^{2j+2}(M^2)}^\epsilon, \forall j, s=1,2,..... $$
Expanding the time derivative $\partial_t^j$ and using $$\|f g\|_{H^r(M^2)}\lesssim \|f \|_{H^r(M^2)} \|g\|_{L^\infty(M^2)}
+ \|g \|_{H^r(M^2)}\|f\|_{L^\infty(M^2)}$$
we are reduced to estimating
$$\|\partial_t^{j_1} u\|_{H^{s}(M^2)}\times \|\partial_t^{j_2} u\|_{L^\infty(M^2)} 
...\times 
\|\partial_t^{j_p} u\|_{L^\infty(M^2)}$$
where $j_1+...+j_p=j$.
Notice that from
$$\|v\|_{L^\infty(M^2)}\lesssim_{\epsilon} \|v\|_{H^1(M^2)}^{1-\epsilon} \|v\|_{H^2(M^2)}^{\epsilon}$$
we get
\begin{align*}\|\partial_t^{j_1} u\|_{H^{s}(M^2)}&\times \|\partial_t^{j_2} u\|_{L^\infty(M^2)} 
...\times 
\|\partial_t^{j_p} u\|_{L^\infty(M^2)}\\\nonumber & \lesssim_{\epsilon} 
\|\partial_t^{j_1} u\|_{H^{s}(M^2)} \times \|\partial_t^{j_2} u\|_{H^1(M^2)}^{1-\epsilon} 
\times \|\partial_t^{j_2} u
\|_{H^2}^{\epsilon}
\times .... \times  \|\partial_t^{j_p} u\|_{H^1(M^2)}^{1-\epsilon} 
\times \|\partial_t^{j_p} u
\|_{H^2(M^2)}^{\epsilon}\end{align*}
and hence by \eqref{eqENimproved2d}
$$...\lesssim \|u\|_{H^{2j_1+s}(M^2)} \times \|u\|_{H^{2j_2+1}(M^2)}^{1-\epsilon} 
\times \|u
\|_{H^{2j_2+2}(M^2)}^{\epsilon}
\times .... \times  \|u\|_{H^{2j_p+1}(M^2)}^{1-\epsilon} 
\times \|u
\|_{H^{2j_p+2}(M^2)}^{\epsilon}$$
$$\lesssim \|u\|_{H^{2j+s}(M^2)}^{\theta_1}  \|u\|_{H^{1}(M^2)}^{1-\theta_1} 
\times \|u\|_{H^{2j+s}}^{\theta_2(1-\epsilon)}  \|u\|_{H^{1}(M^2)}^{(1-\theta_2)(1-\epsilon)}
\times....\times  \|u\|_{H^{2j+s}(M^2)}^{\theta_p
(1-\epsilon)}  \|u\|_{H^{1}(M^2)}^{(1-\theta_p)(1-\epsilon)}
\times \|u
\|_{H^{2j+2}(M^2)}^{\epsilon(p-1)}$$
where at the last step we have used an interpolation argument with
$$\begin{cases}\theta_1 (2j+s) + (1-\theta_1)=2j_1+s\\
\theta_l (2j+s) + (1-\theta_l)={2j_l+1}, \quad l=2,...,p.\end{cases}
$$
Notice that we get $\sum_{l=1}^p \theta_l=1$
and we conclude by \eqref{H1ener}.
\end{proof}
\begin{proposition}\label{theorStrichartz}
We have the following estimate for every $u(t, x)$ solution to \eqref{genericNLS}
for $(p, l)=(3,3)$
and for every $\epsilon>0$, $T\in (0,1)$:
\begin{align}\label{withoutderivative}
\|\partial_t^j u\|_{L^2_TL^{6}(M^3)} \lesssim_{\epsilon, \|\varphi\|_{H^1}} &    
\|\partial_t^j u \|_{L^\infty_T L^{2}(M^3)}^{1-\epsilon}\|\partial_t^j u \|_{L^\infty_T H^{1}
(M^3)}^\epsilon \nonumber\\
 & {}+ \sqrt T \|
u\|_{L^\infty_T H^{2j}(M^3)}^{1/2}
\|u\|_{L^\infty_T H^{2j+1}(M^3)}^{1/2} \\
 & {}+\sqrt T
\sum_{j_1+j_2+j_3=j}
\|u\|_{L^\infty_T H^{2j_1}(M^3)} \|
u\|_{L^\infty_T H^{2j_2+1}(M^3)} \|u\|_{L^\infty_T H^{2j_3+1}(M^3)}
\,,\nonumber
\end{align}
and
\begin{align}\label{derivative}
\|\partial_t^j u\|_{L^2_TW^{1,6}(M^3)} \lesssim_{\epsilon, \|\varphi\|_{H^1}}&    
\|\partial_t^j u \|_{L^\infty_T H^{1}(M^3)}^{1-\epsilon}\|\partial_t^j u \|_{L^\infty_T H^{2}
(M^3)}^\epsilon \nonumber\\  & {}+ \sqrt T \|
u\|_{L^\infty_T H^{2j+1}(M^3)}^{1/2}
\|u\|_{L^\infty_T H^{2j+2}(M^3)}^{1/2} \\\nonumber & {}+\sqrt T
\sum_{j_1+j_2+j_3=j}
\|u\|_{L^\infty_T H^{2j_1+1}(M^3)} \|
u\|_{L^\infty_T H^{2j_2+1}(M^3)} \|u\|_{L^\infty_T H^{2j_3+1}(M^3)}\,. 
\end{align}
\end{proposition}

\begin{proof} We prove \eqref{derivative}, the proof of \eqref{withoutderivative} being similar.
By using Strichartz estimates 
and the equation solved by $\partial_t^j u$ we get:
\begin{align*}\|\partial_t^j u\|_{L^2_T W^{1,6}(M^3)} 
&\lesssim   \|\partial_t^j u \|_{L^\infty_T H^{1+\epsilon}}
+ \sqrt T \|\partial_t^j u\|_{L^\infty_T H^{3/2}(M^3)}
+\|\partial_t^j (u|u|^2)\|_{L^2_T W^{1, 6/5}(M^3)}\\\nonumber
&\lesssim 
\|\partial_t^j u \|_{L^\infty_T H^{1}(M^3)}^{1-\epsilon}\|\partial_t^j u \|_{L^\infty_T H^{2}
(M^3)}^\epsilon
+\|\partial_t^j u\|_{L^\infty_T H^{1}(M^3)}^{1/2}
\|\partial_t^j u
\|_{L^\infty_T H^{2}(M^3)}^{1/2}
\\\nonumber &+\|\partial_t^j (u|u|^2)\|_{L^2_T W^{1, 6/5}(M^3)}.
\end{align*}
Notice that by expanding the time derivative,
and by using H\"older we get
\begin{align*}\|\partial_t^j (u|u|^2)\|_{W^{1, 6/5}(M^3)}
&\lesssim \sum_{j_1+j_2+j_3=j}
\|\partial_t^{j_1} u\|_{H^{1}(M^3)} \|\partial_t^{j_2} u\|_{L^6(M^3)} \|\partial_t^{j_3} u\|_{L^6(M^3)}\\\nonumber& 
\lesssim \sum_{j_1+j_2+j_3=j}
\|\partial_t^{j_1} u\|_{H^{1}(M^3)} \|\partial_t^{j_2} u\|_{H^1(M^3)} \|\partial_t^{j_3} u\|_{H^1(M^3)}.\end{align*}
We then conclude by using Proposition \ref{equiBNN}
in the special case of the cubic NLS on $M^3$.\end{proof}

\section{Polynomial growth of $H^{2k}$ for pure power NLS on $M^2$}

This section is devoted to the proof of Theorem \ref{main2d}
in the case $m=2k$.
We shall need the following estimate.

\begin{proposition}\label{profund}
Let us assume that $u(t,x)$ solves \eqref{genericNLS} 
with $d=2$ and $p=2n+1\geq 3$. Then we have  
the following bound for every $T\in (0,1)$
\begin{align*}
\int_0^T |\hbox{ r.h.s. } \eqref{energy2}| ds
&\lesssim \sqrt T \|u\|_{L^\infty_T H^{2k}}^{\frac{4k-3+2s_0}{2k-1} +\epsilon}+
\|u\|_{L^\infty_T H^{2k}}^{\frac{4k-4}{2k-1}+\epsilon}.
\end{align*}
\end{proposition}

\begin{proof} Since we work on a $2d$ compact manifold we simplify notations as follows:
$L^q, W^{s, q}, H^s$ denote the spaces $L^q(M^2), W^{s, q}(M^2), H^s(M^2)$.
In the sequel we shall also make use of the following
inequality:
\begin{equation}\label{interpeas}
\|u\|_{L^\infty_T H^{s}}\lesssim_{_{\|\varphi\|_{H^1}}} \|u\|_{L^\infty_T H^{2k}}^\frac{s-1}{2k-1}, \quad
\end{equation}
that in turn follows by combining an elementary interpolation inequality 
with \eqref{H1ener}.

Let $I,II,III,IV,V,VI$ be the successive terms on each line of the r.h.s. in
\eqref{energy2}.
Estimating $I$ can be reduced to controling the following terms:
\begin{align}
\label{third1}\int_0^T \|\partial_t^{k_1} u\|_{W^{1,4}}^2 &\|\partial_t^{k_2} u\|^2_{L^\infty} 
\|\partial_t u\|_{L^2} \|u\|^{p-4}_{L^\infty} ds,\\\nonumber
&k_1+k_2=k-1
\end{align}
and we have by Proposition \ref{equiBNN}, Proposition \ref{theorStrichartz2} 
and H\"older inequality:
\begin{align*}\eqref{third1}&\lesssim \sqrt T \|u\|_{L^\infty H^{2k_2+1}}^2 \|u\|_{L^\infty H^{2}}
\|\partial_t^{k_1}u\|_{L^4_T W^{1,4}}^2 
\\\nonumber &\lesssim \sqrt T \|u\|_{L^\infty H^{2k_2+1}}^2 \|u\|_{L^\infty H^{2}}
\|u\|_{L^\infty_T H^{2k_1+1}}^{2(1-s_0)}
\|u\|_{L^\infty_T H^{2k_1+2}}^{2s_0} \|u\|_{L^\infty H^{2k}}^\epsilon\\
\nonumber &\lesssim \sqrt T \|u\|_{L^\infty H^{2k}}^{\frac{4k-3+2s_0 }{2k-1}+\epsilon}.
\end{align*}
where at the last step we have used \eqref{interpeas}. Concerning $II$ we are reduced to controling
\begin{align}\label{third3}
 \int_0^T 
\|\partial_t^{j_1} u\|_{L^2} &\big( \prod_{h=2,..., p-1} \|\partial_t^{j_h} u\|_{L^\infty(M^2)}
\big) \|\partial_t^{k_1} u\|_{W^{1,4}} \|\partial_t^{k_2}
u\|_{W^{1,4}},\\\nonumber
&j_1+...+j_p=k,  \quad k_1+k_2=k-1\,.
\end{align}
By using the interpolation estimate
$$\|v\|_{L^\infty}\lesssim \|v\|_{H^1}^{1-\epsilon} \|v\|_{H^{2k}}^\epsilon$$
together with Proposition \ref{equiBNN}, Proposition \ref{theorStrichartz2} 
and H\"older inequality, we get:
\begin{align*}\eqref{third3} &\lesssim \sqrt T \|u\|_{L^\infty_T H^{2k}}^{\epsilon} \|u\|_{L^\infty_T H^{2j_1}} 
\big( \prod_{h=2,..., p-1} \|u\|_{L^\infty_T H^{2j_h+1}}\big) \\\nonumber \times & 
\|u\|_{L^\infty_T H^{2k_1+1}}^{1-s_0} 
\|u\|_{L^\infty_T H^{2k_1+2}}^{s_0} \|u\|_{L^\infty_T H^{2k_2+1}}^{1-s_0} 
\|u\|_{L^\infty_T H^{2k_2+2}}^{s_0}
\lesssim \sqrt T \|u\|_{L^\infty_T H^{2k}}^{\frac{4k-3+2s_0}{2k-1} +\epsilon},
\end{align*}
where we used \eqref{interpeas} at the last step. Next we deal with $III$ and it is sufficient to control:
\begin{align}\label{fourth}
\int_0^T & \|\partial_t^{h_1} u\|_{L^\infty}
\|\partial_t^{h_2} u\|_{W^{1,4}} 
\|\partial_t^{m_1} u\|_{L^2} \big( \prod_{i=2,...,p-3} \|\partial_t^{m_i} u\|_{L^\infty} \big)
\|\partial_t^{l_1} u\|_{L^\infty} \|\partial_t^{l_2} u\|_{W^{1,4}} ,
\\\nonumber& h_1+h_2=j\in [0, k-1],  \quad m_1+...+m_{p-3}=k-j,
\quad l_1+l_2=k-1\,,
\end{align}
and arguing as above, it can be estimated by:
\begin{align*}\eqref{fourth}&\lesssim  \sqrt T \|u\|_{L^\infty_T H^{2k}}^\epsilon \|u\|_{L^\infty_T H^{2h_1+1}}
\|u\|_{L^\infty_T H^{2h_2+1}}^{1-s_0} \|u\|_{L^\infty_T H^{2h_2+2}}^{s_0}
\|u\|_{L^\infty_T H^{2m_1}} \\\nonumber &\times  \big( \prod_{i=2,...,p-3} \|u\|_{L^\infty
H^{2m_i+1}} \big) \|u\|_{L^\infty_T H^{2l_1+1}}
\|u\|_{L^\infty_T H^{2l_2+1}}^{1-s_0}
\|u\|_{L^\infty_T H^{2l_2+2}}^{s_0} 
\lesssim \sqrt T \|u\|_{L^\infty_T H^{2k}}^{\frac{4k-3+2s_0}{2k-1} +\epsilon}.
\end{align*}
In order to treat $IV$ we are reduced to controling
\begin{align}\label{third5}
\int_0^T  
\|\partial_t^{j_1} u\|_{L^2}
&\big( \prod_{h=2,..., p-1} \|\partial_t^{j_h} u\|_{L^\infty}
\big)\|\partial_t^j (\Delta_g u)\|_{L^4} \|\partial_t^{k-1-j} \bar u \|_{L^4},\\
\nonumber & j_1+...+j_{p-1}=k, 
\end{align}
and by a similar argument as above we have:
\begin{align*}\eqref{third5}
& \lesssim \sqrt T \|u\|_{L^\infty_T H^{2k}}^{\epsilon}
\|u\|_{L^\infty_T H^{2j_1}}
\big( \prod_{h=2,..., p-1} \|u\|_{L^\infty_T H^{2j_h+1}}\big) \\\nonumber &\times 
\|u\|_{L^\infty_T H^{2j+2}}^{1-s_0}\|u\|_{L^\infty_T H^{2j+3}}^{s_0}
\|u\|_{L^\infty_T H^{2k-2-2j}}^{1-s_0}\|u\|_{L^\infty_T H^{2k-2j-1}}^{s_0}
\\\nonumber & \lesssim \sqrt T \|u\|_{L^\infty_T H^{2k}}^{\frac{4k-3+2s_0}{2k-1} +\epsilon}.
\end{align*}
In order to estimate $V$
it is sufficient to control the following terms:
\begin{align}\label{fifth}
\int_0^T \|\partial_t^{m_1} u\|_{L^4} 
&\big(\prod_{i=2,...,p-1} \|\partial_t^{m_i} u\|_{L^\infty}\big )
\|\partial_t^{k-j}u\|_{L^{4}} \|\partial_t^k u\|_{L^{2}}, 
\\\nonumber &
m_1+...+m_{p-1}=j,\end{align}
and a usual we get:
\begin{align*}\eqref{fifth}\lesssim  \sqrt T \|u\|_{L^\infty_T H^{2k}}^{1+\epsilon} 
\|u\|_{L^\infty_T H^{2m_1}}^{1-s_0}&\|u\|_{L^\infty_T H^{2m_1+1}}^{s_0} 
\big(\prod_{i=2,...,p-1} \|u\|_{L^\infty_T H^{2m_i+1}}\big ) 
\|u\|_{L^\infty_T H^{2k-2j}}^{1-s_0}\|u\|_{L^\infty_T H^{2k-2j+1}}^{s_0}
\\\nonumber 
&\lesssim \sqrt T \|u\|_{L^\infty_T H^{2k}}^{\frac{4k-3+2s_0}{2k-1} +\epsilon}
\end{align*}
We conclude with the estimate of $VI$ that in turn can be reduced to controling
\begin{align}\label{sixth1} 
 \int_0^T 
\|\partial_t^{m_1} & u\|_{L^2}
\big( \prod_{i=2,...,p-1} \|\partial_t^{m_i} u\|_{L^\infty}\big)
\| \partial_t^{k-j} u\|_{L^\infty}
\|\partial_t^{l_1} u\|_{L^2}
\big( \prod_{i=2,...,p} \|\partial_t^{l_i} u\|_{L^\infty}\big) , 
\\\nonumber &m_1+...+m_{p-1}=j, \quad l_1+...+l_{p}=k-1,
\end{align}
and
we get
\begin{align*}\eqref{sixth1}&\lesssim T \|u\|_{L^\infty_T H^{2k}}^{\epsilon}
\|u\|_{L^\infty_T H^{2m_1}}
\big( \prod_{i=2,...,p-1} \|u\|_{L^\infty_T H^{2m_i+1}}\big)\\\nonumber
&\times
\|u\|_{L^\infty_T H^{2k-2j+1}} \|u\|_{L^\infty_T H^{2l_1}}
\big( \prod_{i=2,...,p-1} \|u\|_{L^\infty_T H^{2l_i+1}}\big)
\lesssim T \|u\|_{L^\infty_T H^{2k}}^{\frac{4k-4}{2k-1}+\epsilon}\,,
\end{align*}
which ends the proof.\end{proof}

The key estimate to deduce Theorem \ref{main2d} 
is the following one.
\begin{proposition}
Let us assume that $u(t,x)$ solves \eqref{genericNLS} 
with $d=2$ and $p\geq 3$. Then we have  
the following bound for every $T<1$ and for every $\epsilon>0$
$$\|u(T)\|_{H^{2k}}^2 - \|u(0)\|_{H^{2k}}^2 \lesssim \sqrt T \|u\|_{L^\infty_T H^{2k}}^{\frac{4k-3+2s_0}{2k-1} +\epsilon}+
\|u\|_{L^\infty_T H^{2k}}^{\frac{4k-4}{2k-1}+\epsilon}.$$
\end{proposition}

\begin{proof}
We write ${\mathcal E}_{2k} (u)= \|\partial_{t}^k u\|_{L^2}^2 + \mathcal R_{2k} (u)$
where $${\mathcal R}_{2k} (u)=-\frac{p-1}4 \int |\partial^{k-1}_{t} \nabla_g (|u|^2)|_g^2 |u|^{p-3}\hbox{dvol}_g - \int |\partial^{k-1}_{t} (|u|^{p-1} u)|^2 \hbox{dvol}_g.$$
We claim that
\begin{equation}\label{restowk}|{\mathcal R}_{2k} (u(t,x))|\lesssim_{\epsilon} \|u\|_{H^{2k}}^{\frac{4k-4}{2k-1}+\epsilon}
+\|u\|_{L^\infty H^{2k}}^{\frac{4k-6}{2k-1}+\epsilon}.
\end{equation}
In fact notice that arguing as along the proof of Proposition
\ref{profund} we get:
\begin{align*}\int&|\partial^{k-1}_{t} \nabla_g (|u|^2)|_g^2 |u|^{p-3} \hbox{dvol}_g
\lesssim \sum_{k_1+k_2=k-1}
\|\partial_t^{k_1} u\|_{W^{1,2}}^2 \|\partial_t^{k_2} u\|^2_{L^\infty} 
\|u\|^{p-3}_{L^\infty}\\\nonumber & \lesssim 
\sum_{k_1+k_2=k-1} \|u\|_{H^{2k_1+1}}^2 \|u\|_{H^{2k_2+1}}^2 
\|u\|_{H^{2k}}^\epsilon\lesssim \|u\|_{H^{2k}}^{\frac{4k-4}{2k-1}+\epsilon}
\end{align*}
and also
\begin{align*}\int&|\partial^{k-1}_{t} (|u|^{p-1} u)|^2 \hbox{dvol}_g
\lesssim \sum_{j_1+....+j_p=k-1}
\|\partial_t^{j_1} u\|_{L^2(M^2)}^2 \big( \prod_{h=1,..., p} \|\partial_t^{j_h} u\|_{L^\infty(M^2)}^2
\big)\\\nonumber 
& \lesssim \sum_{j_1+....+j_p=k-1}\|u\|_{H^{2j_1}}^2 \big( \prod_{h=1,..., p} \|u\|_{H^{2j_h+1}(M^2)}^2
\big) \|u\|_{L^\infty H^{2k}}^\epsilon\lesssim 
\|u\|_{L^\infty H^{2k}}^{\frac{4k-6}{2k-1}+\epsilon}.
\end{align*}
Next notice that if we integrate the identity 
\eqref{energy2} and we use Proposition \ref{profund}
then
$$\|\partial_t^{k} u(T)\|_{L^2}^2 - \|\partial_t^{k} u(0)\|_{L^2}^2
\lesssim \sup_{(0,T)} |\mathcal{R}_{2k} (u)|  + 
\sqrt T \|u\|_{L^\infty_T H^{2k}}^{\frac{4k-3+2s_0}{2k-1} +\epsilon}+
\|u\|_{L^\infty_T H^{2k}}^{\frac{4k-4}{2k-1}+\epsilon}.$$
We conclude by \eqref{restowk} and Proposition \ref{equiBNN}.
\end{proof}

\section{Exponential growth for $H^{2k}$ norms of solutions to cubic NLS on $M^3$}

The aim of this section is the proof of Theorem \ref{main3d} in the case $m=2k$.

The following is the analogue version of Proposition \ref{profund} in $3d$ for the cubic NLS.
\begin{proposition}\label{profund3}
Let us assume that $u(t,x)$ solves \eqref{genericNLS} 
with $d=3$ and $p=3$. Then we have  
the following bound for every $T\in (0, 1)$
\begin{align*}
\int_0^T |\hbox{ r.h.s. } \eqref{energy2}| ds
&\lesssim T \|u\|_{L^\infty_T H^{2k}}^2 + \|u\|_{L^\infty_T H^{2k}}^\gamma
\end{align*}
for some $\gamma\in (0,2)$.
\end{proposition}
\begin{proof}
 Since we work on a $3d$ compact manifold we simplify the notations as follows:
$L^q, W^{s, q}, H^s$ denote the spaces $L^q(M^3), W^{s, q}(M^3), H^s(M^3)$.
In the sequel we shall also make use of the following
inequalities:
\begin{align}\label{interpeas3}
\|u\|_{L^\infty_T H^{s}}&\lesssim_{_{\|\varphi\|_{H^1}}} \|u\|_{L^\infty_T H^{2k}}^\frac{s-1}{2k-1},  \quad
\end{align}
that in turn follow by combining an elementary interpolation inequality 
with \eqref{H1ener}.
We also notice that by combining Proposition \ref{equiBNN}
and Proposition \ref{theorStrichartz} with \eqref{interpeas3}
we get:
\begin{align}\label{L2L6}
\|\partial_t^j u\|_{L^2_TL^{6}} \lesssim_\epsilon  & 
\|u \|_{L^\infty_T H^{2j}}^{1-\epsilon}\|u\|_{L^\infty_T H^{2j+1}}^\epsilon+  
\sqrt T \|
u\|_{L^\infty_T H^{2j}}^{1/2}
\|u\|_{L^\infty_T H^{2j+1}}^{1/2} \\\nonumber & +\sqrt T
\sum_{j_1+j_2+j_3=j}
\|u\|_{L^\infty_T H^{2j_1}} \|u\|_{L^\infty_T H^{2j_2+1}} 
\|u\|_{L^\infty_T H^{2j_3+1}}\\\nonumber
&\lesssim 
\|u\|_{L^\infty_T H^{2k}}^{\frac{2j-1+\epsilon}{2k-1} }+ \sqrt T \|u\|_{L^\infty_T H^{2k}}^{\frac{4j-1}{4k-2}}
+ \sqrt T \|u\|_{L^\infty_T H^{2k}}^{\frac{2j-1}{2k-1}}
\end{align}
and 
\begin{align}\label{L2W16}
\|\partial_t^j u\|_{L^2_TW^{1,6}} \lesssim_\epsilon  & 
\|u \|_{L^\infty_T H^{2j+1}}^{1-\epsilon}\|u\|_{L^\infty_T H^{2j+2}}^\epsilon+  
\sqrt T \|
u\|_{L^\infty_T H^{2j+1}}^{1/2}
\|u\|_{L^\infty_T H^{2j+2}}^{1/2} \\\nonumber & +\sqrt T
\sum_{j_1+j_2+j_3=j}
\|u\|_{L^\infty_T H^{2j_1+1}} \|u\|_{L^\infty_T H^{2j_2+1}} 
\|u\|_{L^\infty_T H^{2j_3+1}}\\\nonumber
&\lesssim 
\|u\|_{L^\infty_T H^{2k}}^{\frac{2j+\epsilon}{2k-1} }+ \sqrt T \|u\|_{L^\infty_T H^{2k}}^{\frac{4j+1}{4k-2}}
+ \sqrt T \|u\|_{L^\infty_T H^{2k}}^{\frac{2j}{2k-1}}
\end{align}
We denote by $I,II,III,IV$ the four  terms on each line of the r.h.s. 
in  \eqref{energy3}.
We first estimate the term $I$.
By developing the time derivatives $\partial_t^k$ and 
$\partial_t^{k-1}$, and by using the H\"older inequality
we are reduced  to estimating:
\begin{equation}\label{prim4}
\int_0^T \|\partial_t^{k_1} u\|_{L^2} \|\partial_t^{k_2}u\|_{L^6} 
\|\partial_t^{j_1} u\|_{W^{1,6}} \|\partial_t^{j_2} u\|_{W^{1,6}} ds,
\end{equation}
\begin{equation*} j_1+j_2=k-1,  k_1+k_2=k.
\end{equation*}
Notice that we have by combining the Sobolev embedding $H^1(M^3)\subset L^6(M^3)$
with
Proposition \ref{equiBNN} for $d=3$ and $p=3$, and \eqref{interpeas3}
\begin{align*}\eqref{prim4}&\lesssim \|u\|_{L^\infty_T H^{2k_1}} \|u\|_{L^\infty_T H^{2k_2+1}} 
\|\partial_t^{j_1}u\|_{L^2_T W^{1,6}} \|\partial_t^{j_2}u\|_{L^2_T 
W^{1,6}}\\\nonumber
&\lesssim \|u\|_{L^\infty_T H^{2k}} \|\partial_t^{j_1}u\|_{L^2_T W^{1,6}} \|\partial_t^{j_2}u\|_{L^2_T 
W^{1,6}}
\end{align*}
and we can continue the estimate by using \eqref{L2W16}.
Indeed we should estimate $\|\partial_t^{j}u\|_{L^2_T W^{1,6}}$
by three terms on the r.h.s. in \eqref{L2W16}.
However we can consider only the term that gives the worse growth
w.r.t. to the power of $\|u\|_{L^\infty H^{2k}}$ (i.e. only the second term
on the r.h.s. of \eqref{L2W16} and all the other terms give a smaller power of
 $\|u\|_{L^\infty H^{2k}}$).
Summarizing we get
 \begin{align*}\eqref{prim4}&
 \lesssim T \|u\|_{L^\infty_T H^{2k}}^2 + \|u\|_{L^\infty_T H^{2k}}^{\gamma}
\end{align*}
for a suitable $\gamma\in (0, 2)$. Next we estimate the term II that can be reduced to estimate the following terms:
\begin{align}
\label{second2} \int_0^T \|\partial_t^{k_1} u\|_{L^2} &\|
\partial_t^{k_2}u\|_{L^6} \|\partial_t^{j} \Delta_g u\|_{L^6}
\|\partial_t^{k-1-j}u\|_{L^6},\\\nonumber
&j=0,...,k-2; \quad k_1+ k_2=k.
\end{align}
By using the Sobolev embedding $H^1(M^3)\subset L^6(M^3)$
in conjunction with Proposition 
\ref{equiBNN} we get
$$\eqref{second2}\lesssim \|u\|_{L^\infty_T H^{2k_1}} 
\|\partial_t^{k_2}u\|_{L^2_T L^{6}} 
\|u\|_{L^\infty_TH^{2j+3}}
\|\partial_t^{k-1-j}u\|_{L^2_T L^6}$$
By using \eqref{L2L6} and \eqref{interpeas3} we get:
$$\eqref{second2}\lesssim \|u\|_{L^\infty_T H^{2k_1}} 
\|u\|_{L^\infty_T H^{2k}}^{\frac{4k_2-1}{4k-2}} \|u\|_{L^\infty_TH^{2j+3}}
\|u\|_{L^\infty_T H^{2k}}^{\frac{4(k-1-j)-1}{4k-2}} 
\lesssim T \|u\|_{L^\infty_T H^{2k}}^2 +  \|u\|_{L^\infty_T H^{2k}}^\gamma,$$
where $\gamma\in (0, 2)$. Concerning the term III we are reduced to 
\begin{align}
\label{terzo2}
\int_0^T \|\partial_t^{j_1} u\|_{L^\infty} &\|\partial_t^{j_2} u\|_{L^\infty} 
\|\partial_t^{k-j} u\|_{L^2}\|\partial_t^{k_1} u\|_{L^6} 
\|\partial_t^{k_2} u\|_{L^6}  \|\partial_t^{k_3} u\|_{L^6}, \\\nonumber& 
j_1+j_2=j, \quad 0\leq j \leq k-1, \quad k_1+k_2+k_3=k-1.
\end{align}
By the Sobolev embedding $H^1(M^3)\subset L^6(M^3)$ and
$H^2(M^3)\subset L^\infty(M^3)$
and Proposition \ref{equiBNN}
we get:
$$\eqref{terzo2}\lesssim \|u\|_{L^\infty_T H^{2j_1+2}}
\|u\|_{L^\infty_T H^{2j_2+2}} \|u\|_{L^\infty_T H^{2k-2j}}
\|\partial_t^{k_1} u\|_{L^2_T L^6}\|\partial_t^{k_2}  u\|_{L^2_T L^6}\|u\|_{L^\infty_T H^{2k_3+1}}.$$
By combining \eqref{L2L6} with \eqref{interpeas3} we get
$$\eqref{terzo2}\lesssim T \|u\|_{L^\infty_TH^{2k}}^2+\|u\|_{L^\infty_TH^{2k}}^{\gamma}$$
for $\gamma\in (0, 2)$. Concerning IV  it is sufficient to estimate
\begin{align}\label{quarto}
\int_0^T \|\partial_t^k u\|_{L^2} & \|\partial_t^{k-j} u\|_{L^6}
 \|\partial_t^{j_1} u\|_{L^6} 
\|\partial_t^{j_2} u\|_{L^6},\\\nonumber
j_1+j_2& =j, \quad 1\leq j \leq k-1.
\end{align}
We can control it by using $H^1(M^3)\subset L^6(M^3)$
and Proposition \ref{equiBNN}:
$$\eqref{quarto}\lesssim  \|u\|_{L^\infty_T H^{2k}} \|u\|_{L^\infty_T H^{2k-2j +1}}
 \|\partial_t^{j_1} u\|_{L^2_T L^6} \|\partial_t^{j_1} u\|_{L^2_TL^6}
$$  
and again by \eqref{L2L6} and \eqref{interpeas3} we get
$$\eqref{quarto}\lesssim T \|u\|_{L^\infty_T H^{2k}}^2 + \|u\|_{L^\infty_T H^{2k}}^\gamma$$
for some $\gamma\in (0,2)$. \end{proof}

\section{Polynomial growth of $H^2$ for sub cubic NLS on $M^3$}
Next we prove Theorem \ref{main3dsub}.
We introduce the following energy
$${\mathcal F}_2 (v(t,x))=
\int_{M^3} |\partial_t v|^2 \hbox{dvol}_g-(p-1) 
\int_{M^3} |v|^{p-1} |\nabla_g |v||^2 \hbox{dvol}_g-\frac{p-1}p \int_{M^3} |v|^{2p}
\hbox{dvol}_g.$$
\begin{proposition}
Let $u(t,x)$ be solution to \eqref{genericNLS} for $d=3$ and $2<p<3$,
then we have
\begin{multline}\label{F2}
\frac d{dt} {\mathcal F}_2 u(t,x) =(p-1)(p-3) \int_{M^3} |u|^{p-2}
\partial_t |u|  |\nabla_g |u||^2 \hbox{dvol}_g\\
{}+ 2 (p-1) \int_{M^3} |u|^{p-2} \partial_t |u|  |\nabla_g u|_g^2 \hbox{dvol}_g.\end{multline}
\end{proposition}

\begin{proof}
We start with the following computation
\begin{align*}
\frac{d}{dt} \| \partial_{t} u\|^{2}_{L^2} & = 2 \Re (\partial^{2}_{t} u , \partial_{t} u) = 2 \Re (\partial_{t} (-\Delta_g u+|u|^{p-1} u), i\partial_{t} u) \\
& = 2 \Im \int_{M^3} (\partial_{t} \nabla_g u, \partial_{t} \nabla_g u)_g
\hbox{dvol}_g
+ 2\Re(\partial_{t} (|u|^{p-1} u),i\partial_{t}u)
\end{align*}
where $(f, g)=\int_{M^3} f \bar g\hbox{dvol}_g$.
Since the first term vanishes we get 
\begin{align*}
\frac{d}{dt} \| \partial_{t} u\|^{2}_{L^2} & =   2\Re(\partial_{t} (|u|^{p-1})  u,i\partial_{t}u)+ 2\Re(|u|^{p-1} \partial_{t}  u,
i\partial_{t}u)\\\nonumber &=
2\Re(\partial_{t} (|u|^{p-1})  u,-\Delta_g u)+2\Re(\partial_{t} (|u|^{p-1})  u,
|u|^{p-1}u))\\\nonumber
 & =  2\Re(\partial_{t} (|u|^{p-1})  u,-\Delta_g u)+
 \frac{p-1}p \frac d{dt} \int_{M^3} |u|^{2p} \hbox{dvol}_g.\end{align*}
By using the identity
$$
\Delta_g (|u|^2)= u \Delta_g \bar u+\bar u \Delta_g u+ 2 |\nabla_g u|_g^2
$$
we get
\begin{align*}
2\Re(\partial_{t} (|u|^{p-1})  u,-\Delta_g u)&= - (\partial_t |u|^{p-1}, \Delta_g |u|^2)+ 2 (\partial_t |u|^{p-1}, |\nabla_g u|_g^2)+
\\\nonumber
&= (\partial_t \nabla_g |u|^{p-1}, \nabla_g |u|^2)+ 2 (\partial_t |u|^{p-1}, |\nabla_g u|_g^2)+
\\\nonumber
&= 2 (p-1) (\partial_t (|u|^{p-2} 
\nabla_g |u|), |u| \nabla_g |u|)+ 2 (\partial_t |u|^{p-1}, |\nabla_g u|_g^2)
\\\nonumber
&= 2 (p-1) \frac d{dt} (|u|^{p-2} 
\nabla_g |u|, |u| \nabla_g |u|)- 2 (p-1)(|u|^{p-2} 
\nabla_g |u|, \partial_t |u| \nabla_g |u|)\\\nonumber&-2 (p-1)(|u|^{p-2} 
\nabla_g |u|, |u| \nabla_g \partial_t |u|)
+ 2 (\partial_t |u|^{p-1}, |\nabla_g u|_g^2)\\\nonumber
&= 2 (p-1) \frac d{dt} (|u|^{p-2} 
\nabla_g |u|, |u| \nabla_g |u|)- 2 (p-1)(|u|^{p-2} 
\nabla_g |u|, \partial_t |u| \nabla_g |u|)\\\nonumber
&-(p-1) 
\frac d{dt}(|u|^{p-1}, |\nabla_g |u||^2)+(p-1) 
(\partial_t |u|^{p-1}, |\nabla_g |u||^2)+ 2 (\partial_t |u|^{p-1}, |\nabla_g u|_g^2)\,,
\end{align*}
which ends the proof.\end{proof}

The following proposition is the analogue version of Proposition \ref{profund3}
in the subcubic case.
\begin{proposition}
We have for every $T\in (0,1)$
$$\int_0^T |r.h.s. \eqref{F2}| ds\lesssim 
T\|u\|_{L^\infty_T H^2}^\frac{p+5}4+ \|u\|_{L^\infty_T H^2}^\gamma$$
for some $\gamma\in (0, \frac{p+5}4).$
\end{proposition}

\begin{proof} We can write the terms on the r.h.s. in \eqref{F2} as $I, II$.
We estimate $I$ and the estimate of $II$ is similar.
We estimate $I$ as follows (we shall use the diamagnetic inequality in order to 
remove $|.|$ inside the derivatives $\nabla_g$ and $\partial_t$)
by the H\"older inequality:
$$|I|\lesssim \|\partial_t u\|_{L^\infty_T L^2} \|u\|_{L^2_T W^{1,\frac{12}{5-p}}}^2 
\|u\|_{L^{6}}^{p-2}\lesssim T^{\frac{6-2p}8}\|\partial_t u\|_{L^\infty_T L^2} \|u\|_{L^\frac{8}{p+1}_T W^{1,\frac{12}{5-p}}}^2$$
where the couple
$\big (\frac 8{p+1}, \frac{12}{5-p}\big )$ is Strichartz admissible.
Notice that by using the equation solved by $u(t,x)$ we are allowed to 
replace $\|\partial_t u\|_{L^\infty_T L^2}$ with $\|u\|_{L^\infty_T H^2}$ and hence
$$|I|\lesssim T^{\frac{6-2p}8}\|u\|_{L^\infty_T H^2} \|u\|_{L^\frac{8}{p+1}_T W^{1,\frac{12}{5-p}}}^2.$$
Next notice that we have the following bound:
$$\|u\|_{L^\frac{8}{p+1}_T W^{1,\frac{12}{5-p}}}\lesssim \|u\|_{L^\infty_T H^1}^{\frac{3-p}{4}}
\|u\|_{L^2_T W^{1,6}}^\frac{p+1}4$$ and hence due to the conservation of the energy
we can continue the estimate above as follows:
$$|I|\lesssim T^{\frac{6-2p}8}\|u\|_{L^\infty_T H^2}\|u\|_{L^2_T W^{1,6}}^\frac{p+1}2$$
We can continue the estimate
by 
using the Strichartz estimates \eqref{L2W16} for $j=0$ 
(which are still available for solutions to subcubic NLS):
$$|I|\lesssim T
\|u\|_{L^\infty_T H^2}\|u\|_{L^\infty_T H^2}^\frac{p+1}4+ \|u\|_{L^\infty_T H^2}^\gamma$$
for some $\gamma\in (0, \frac{p+5}4)$
(indeed we have estimated the term 
$\|u\|_{L^2_T W^{1,6}}$ with the middle term on the r.h.s. in 
\eqref{L2W16} since it is the one that involves the larger power of $\|u\|_{L^\infty_TH^2}$,
and the lower power are absorbed in the term $\|u\|_{L^\infty_T H^2}^\gamma$).
\end{proof}

The proof of Theorem \ref{main3dsub} can be concluded 
easily by integrating the identity \eqref{F2} on $[0, T]$ and arguing
exactly as along the proof of Theorem \ref{main2d} and \ref{main3d}.

\section{Growth of Odd Sobolev norms $H^{2k+1}$}

The proofs of Theorems \ref{main2d} and \ref{main3d} 
(which have been proved in the case $m=2k$) can be adapted to the
case $m=2k+1$ by using the following modified energies:
\begin{align*}
{\mathcal E}_{2k+1}(u)&= \frac 12
\| \partial_{t}^{k}\nabla_g u\|^{2}_{L^2}
+\frac 12 \int |u|^{p-1} 
|\partial_{t}^{k}u|^2 \hbox{dvol}_g +\frac{p-1}8 \int |u|^{p-3}
 |\partial_{t}^{k} (|u|^2)|^2 \hbox{dvol}_g
 \\\nonumber -\Re \sum_{j=1}^{k-1} c_j \int
 \partial_t^j u &\partial_t^{k-j} (|u|^{p-1})
\partial_t^{k} \bar u \hbox{dvol}_g 
-\sum_{j=1}^{k-1} c_j \int \partial_t^{k-j} (|u|^{p-3})
\partial_{t}^{j}(|u|^2)\partial_{t}^{k} (|u|^2) \hbox{dvol}_g.
\end{align*}
Indeed we have the following proposition, from which ones may conclude
the proof of Theorems \ref{main2d} and \ref{main3d} 
in the case $m=2k+1$, exactly as we did in the case $m=2k$. We leave details to the reader.
\begin{proposition}
Let $u(t, x)$ be solution to \eqref{eq:1} with initial data $\varphi\in H^{2k+1}$,
then we have the following identity
\begin{align*}
\frac d{dt} {\mathcal E}_{2k+1}(u(t,x))= &\frac 12 \int \partial_t (|u|^{p-1})  |\partial_{t}^{k}u|^2 \hbox{dvol}_g
\\ & {}- \Re \sum_{j=1}^{k-1} c_j \int \partial_t^{j+1} u \partial_t^{k-j} (|u|^{p-1})
\partial_t^{k} \bar u \hbox{dvol}_g \nonumber
\\\nonumber
 & {}-\Re \sum_{j=1}^{k-1} c_j \int \partial_t^j u \partial_t^{k-j+1} (|u|^{p-1})
\partial_t^{k} \bar u \hbox{dvol}_g
 \nonumber\\
& {}+\frac{p-1}8 \int_{M^2} \partial_t (|u|^{p-3})
 |\partial_{t}^{k} (|u|^2)|^2 \hbox{dvol}_g
\nonumber \\
& {}+ \sum_{j=1}^{k-1} c_j \int \partial_t^{k-j+1} (|u|^{p-3})
\partial_{t}^{j}(|u|^2)\partial_{t}^{k} (|u|^2) \hbox{dvol}_g\nonumber \\
& {}+\sum_{j=1}^{k-1} c_j \int \partial_t^{k-j} (|u|^{p-3})
\partial_{t}^{j+1}(|u|^2)\partial_{t}^{k} (|u|^2) \hbox{dvol}_g\nonumber\\
& {}+\sum_{j=1}^k c_j \int \partial_{t}^{k}(|u|^{p-1}) 
\partial_t^j u \partial_t^{k+1-j} \bar u \hbox{dvol}_g,\nonumber
\end{align*}
where $c_j\in \R$ are suitable real numbers that can change in different lines.
\end{proposition}
\begin{proof}
First of all notice that we have
\begin{align*}\Re(i \partial_{t}^{k+1} u, \partial^{k}_{t} u) & =   \Re(\partial_{t}^{k}(-\Delta_g u), \partial^{k}_{t} u)+\Re(\partial_{t}^{k}(u|u|^{p-1}), \partial^{k}_{t} u)  \\
     & =   \| \partial_{t}^{k}\nabla_g u\|^{2}_{L^2}+\Re(\partial_{t}^{k}(u|u|^{p-1}), 
\partial^{k}_{t} u).
 \end{align*}
Due to the identity above and by taking time derivative we get:
\begin{align*}
\frac d{dt} \big (\| \partial_{t}^{k}\nabla_g u\|^{2}_{L^2}+\Re(\partial_{t}^{k}(u|u|^{p-1}), \partial^{k}_{t} u)\big) = &
\frac{d}{dt} \Re(i \partial_{t}^{k+1} u, \partial^{k}_{t} u) 
\\
 = &  \Re( i \partial_{t}^{k+2} u, \partial_{t}^{k} u) \\
  = & \Re( \partial_{t}^{k+1} (-\Delta_g u) , \partial_{t}^{k} u) +\Re( \partial_{t}^{k+1}(|u|^{p-1}u), \partial_{t}^{k} u)\\\nonumber = & \frac 12 \frac d{dt} \|\partial_t^k \nabla_g u\|_{L^2}^2 
 +\Re( \partial_{t}^{k+1}(|u|^{p-1}u), \partial_{t}^{k} u).\end{align*}
 Next we focus on the second term on the r.h.s.
 \begin{align*}
 \Re( \partial_{t}^{k+1}(|u|^{p-1}u), \partial_{t}^{k} u)
 =\frac d{dt} \Re( \partial_{t}^{k}(|u|^{p-1}u), \partial_{t}^{k} u) &-\Re( \partial_{t}^{k}(|u|^{p-1}u), \partial_{t}^{k+1} u) \\\nonumber= 
 \frac d{dt} \Re( \partial_{t}^{k}(|u|^{p-1}u), \partial_{t}^{k} u) -\Re( \partial_{t}^{k}(|u|^{p-1})u, \partial_{t}^{k+1} u)&-\Re( |u|^{p-1} \partial_{t}^{k}u, \partial_{t}^{k+1} u)
 \\\nonumber
+ \Re \sum_{j=1}^{k-1} c_j (\partial_t^j u \partial_t^{k-j} (|u|^{p-1}),
\partial_t^{k+1} u)
= 
\frac d{dt} \Re( \partial_{t}^{k}(|u|^{p-1}u), \partial_{t}^{k} u) 
&-\Re( \partial_{t}^{k}(|u|^{p-1})u, \partial_{t}^{k+1} u)\\\nonumber
- \frac 12 \frac d{dt}\int |u|^{p-1} 
|\partial_{t}^{k}u|^2 \hbox{dvol}_g  +\frac 12 \int \partial_t (|u|^{p-1}) 
|\partial_{t}^{k}u|^2 \hbox{dvol}_g
&+ \Re \sum_{j=1}^{k-1} c_j (\partial_t^j u \partial_t^{k-j} (|u|^{p-1}),
\partial_t^{k+1} u)\\\nonumber
= 
\frac d{dt} \Re( \partial_{t}^{k}(|u|^{p-1}u), \partial_{t}^{k} u) 
-\Re( \partial_{t}^{k}(|u|^{p-1})u, \partial_{t}^{k+1} u)
&- \frac 12 \frac d{dt}\int |u|^{p-1} 
|\partial_{t}^{k}u|^2 \hbox{dvol}_g \\\nonumber +\frac 12 \int \partial_t (|u|^{p-1}) 
|\partial_{t}^{k}u|^2 \hbox{dvol}_g
&+ \frac d{dt} \Re \sum_{j=1}^{k-1} c_j (\partial_t^j u \partial_t^{k-j} (|u|^{p-1}),
\partial_t^{k} u) \\\nonumber 
- \Re \sum_{j=1}^{k-1} c_j (\partial_t^{j+1} u \partial_t^{k-j} (|u|^{p-1}),
\partial_t^{k} u)&-\Re \sum_{j=1}^{k-1} c_j (\partial_t^j u \partial_t^{k-j+1} (|u|^{p-1}),
\partial_t^{k} u). 
\end{align*}
Next we deal with the second term on the r.h.s.:
\begin{align*}-\Re( \partial_{t}^{k}(|u|^{p-1})u, \partial_{t}^{k+1} u)
&=-\frac 12 \int \partial_{t}^{k}(|u|^{p-1})\partial_{t}^{k+1} (|u|^2) \hbox{dvol}_g
+\sum_{j=1}^k c_j \int \partial_{t}^{k}(|u|^{p-1}) \partial_t^j u \partial_t^{k+1-j} \bar u 
\hbox{dvol}_g
\end{align*}
and we notice that
$\partial_{t}^{k}(|u|^{p-1})=\frac{p-1}2 
\partial_t^{k-1} (\partial_t (|u|^2) |u|^{p-3}).$ 
Hence we can continue the identity above as follows
\begin{align*}
...&=-\frac{p-1}4 \int |u|^{p-3}
\partial_{t}^{k}(|u|^2)\partial_{t}^{k+1} (|u|^2) \hbox{dvol}_g
+ \sum_{j=1}^{k-1} c_j \int \partial_t^{k-j} (|u|^{p-3})
\partial_{t}^{j}(|u|^2)\partial_{t}^{k+1} (|u|^2) \hbox{dvol}_g
\\\nonumber& +\sum_{j=1}^k c_j \int \partial_{t}^{k}(|u|^{p-1}) \partial_t^j u \partial_t^{k+1-j} \bar u \hbox{dvol}_g \\\nonumber
&=-\frac{p-1}8 \frac d{dt} \int |u|^{p-3}
 |\partial_{t}^{k} (|u|^2)|^2 \hbox{dvol}_g +\frac{p-1}8 \int \partial_t (|u|^{p-3})
 |\partial_{t}^{k} (|u|^2)|^2 \hbox{dvol}_g
\\\nonumber&+ \sum_{j=1}^{k-1} c_j \int \partial_t^{k-j} (|u|^{p-3})
\partial_{t}^{j}(|u|^2)\partial_{t}^{k+1} (|u|^2) \hbox{dvol}_g
 +\sum_{j=1}^k c_j \int \partial_{t}^{k}(|u|^{p-1}) \partial_t^j u \partial_t^{k+1-j} \bar u
\hbox{dvol}_g
\end{align*}
and by elementary considerations
\begin{align*}
...&=-\frac{p-1}8 \frac d{dt} \int |u|^{p-3}
 |\partial_{t}^{k} (|u|^2)|^2 \hbox{dvol}_g +\frac{p-1}8 \int \partial_t (|u|^{p-3})
 |\partial_{t}^{k} (|u|^2)|^2 \hbox{dvol}_g
\\\nonumber&+ \frac d{dt} \sum_{j=1}^{k-1} c_j \int \partial_t^{k-j} (|u|^{p-3})
\partial_{t}^{j}(|u|^2)\partial_{t}^{k} (|u|^2) \hbox{dvol}_g
+\sum_{j=1}^{k-1} c_j \int \partial_t^{k-j+1} (|u|^{p-3})
\partial_{t}^{j}(|u|^2)\partial_{t}^{k} (|u|^2) \hbox{dvol}_g
\\\nonumber &+\sum_{j=1}^{k-1} c_j \int \partial_t^{k-j} (|u|^{p-3})
\partial_{t}^{j+1}(|u|^2)\partial_{t}^{k} (|u|^2) \hbox{dvol}_g
+\sum_{j=1}^k c_j \int\partial_{t}^{k}(|u|^{p-1}) \partial_t^{j} 
u \partial_t^{k+1-j} \bar u
\hbox{dvol}_g.
\end{align*}
The proof is complete.
\end{proof}

\end{document}